\documentclass[12pt,a4paper]{amsart}
\usepackage{amssymb,amsmath,amsthm,amscd,amsfonts}

\newtheorem{theorem}{Theorem}[section]

\newtheorem{proposition}[theorem]{Proposition}
\newtheorem{corollary}[theorem]{Corollary}
\theoremstyle{definition}
\newtheorem{definition}[theorem]{Definition}
\newtheorem{example}[theorem]{Example}

\newtheorem{remark}[theorem]{Remark}
\numberwithin{equation}{section}

\begin{document}

\title[Coproduct Cancellation on \textbf{Act}-$S$]{Coproduct Cancellation on \textbf{Act}-$S$}

\author[Kamal Ahmadi and Ali Madanshekaf ]{Kamal Ahmadi$^1$, Ali Madanshekaf $^2$}

\address{$^{1}$ Department of Mathematics, Semnan  University, P. O. Box 35131-19111, Semnan, Iran.}

\email{kamal.ahmadi.math@gmail.com}

\address{$^{2}$ Department of Mathematics, Semnan  University, P. O. Box 35131-19111, Semnan, Iran.}

\email{amadanshekaf@semnan.ac.ir}

\begin{abstract}
The themes of cancellation, internal
cancellation, substitution have led to a lot of interesting
research in the theory of modules over commutative and
noncommutative rings. In this paper, we introduce and study cancellation problem in the theory of acts over monoids. We show that if $A$ is an $S$-act and $A={\dot\bigcup_{i\in I}}A_i$ is the unique
decomposition of $A$ into indecomposable subacts $A_i, i\in I$ such that the set
$P=\{{\rm Card} [i] \mid i\in I\}$ is finite, then $A$ is
cancellable if and only if the equivalence class $[i]=\{j\in I \mid  A_i\cong A_j\}$ is finite, for every $i\in I$.
Likewise, we prove that  every $S$-act is cancellable if and only if it  is
internally cancellable. Thus, the concepts cancellation and internal cancellation coincide here.
\end{abstract}
\maketitle
\section{Introduction and Preliminaries}
Jonsson and Tarski  were considered cancellation problem
 initiatory in 1947 (see~\cite{tar}). In the study of any algebraic
system in which there is a notion of a direct sum, the theme of
cancellation arises very naturally: if $A\oplus B\cong A\oplus C$ in
the given system, can we conclude that $B\cong C$? The answer is,
perhaps not surprisingly, sometimes ``yes'' and sometimes ``no'': it
all depends on the algebraic system, and it depends heavily on the
choice of $A$ as well.

Importance of cancellation problem is obvious, since Serre's
famous conjecture on the freeness of f.g. projective modules over
a polynomial ring $R=K[x_1,\cdots ,x_n]$ (for a field $K$) boiled down
to a statement about the cancellability of R (see \cite{ca} and
\cite{lam3}).

Starting with a simple example, we all know that, by the Fundamental
Theorem of Abelian Groups, the category of finitely generated
abelian groups satisfies cancellation. If $A$ is a finitely generated
abelian group, then for any abelian groups $B$ and $C$, $A\oplus
B\cong A\oplus C$ still implies $B\cong C$. There exists many
torsionfree abelian groups of rank 1 that are not cancellable in the
category of torsionfree abelian groups of finite rank, according to
\cite{jon}.

Now let $R$ be an associative ring with an identity. We say that
$R$ has stable range one provided that $aR+bR=R$ with $a,b\in R$
implies that there exists some $y\in R$ such that $a+by$ is unit.
If ${\rm End}_R(A)$ has stable range one, then $A\oplus B\cong A\oplus
C$ implies that $B\cong C$ for any right $R$-modules $B,C$ (see
 \cite[Theorem 2]{evans}). Since every local ring has stable
range one, therefore every strongly indecomposable module $A$
(that is, ${\rm End}(A)$ is local) is cancellable and every simple
module is cancellable by ~\cite[Lemma 4.13.3]{ander}.

On the other hand, in categories of modules over rings there are several variations on
the notion of cancellation. For instance, for given module $A$,
$A=K\oplus N=K'\oplus N'$ with $N\cong N'$, does it follow that
$K\cong K'?$ If the answer is always ``yes'', $A$ is said to satisfy
internal cancellation (or $A$ is internally cancellable). Another
variations are the ``Substitution'' and ``Dedekind-finite'' properties (see~\cite{lam1} for the definitions).  These properties are easily seen to be
related as follows:
\begin{center}
 Substitution
$\Rightarrow$Cancellation $\Rightarrow$ Internal cancellation
$\Rightarrow$ Dedekind-finite.
\end{center}
We encourage the reader to consult~\cite{ander, lam2, lam1} about cancellation problem on the category of modules on arbitrary rings.

  Let $S$ be a monoid with identity $1$. Recall that
a (right) $S$-act $A$ is a non-empty set equipped with a map $\lambda:
A\times S\to A$ called its action, such that, denoting $\lambda(a,
s)$ by $as$, we have $a1 = a$ and $a(st)=(as)t$, for all $a\in A$,
and $s, t\in S$. The category of all $S$-acts, with
action-preserving ($S$-act) maps ($f : A\to B$ with $f(as) =
f(a)s$, for $s\in S$, $a\in A$)  is denoted by \textbf{Act}-$S$.
Clearly $S$ itself is an $S$-act with its operation as the action.
Throughout this paper, all $S$-acts will be right $S$-act.

Recall that the category \textbf{Act}-$S$ has coproducts of any non-empty families
of $S$-acts. More Precisely, if   $I$ is a non-empty set and  $X_i\in\textrm{\bf
Act}-S, i\in I$ then by~\cite[Proposition 2.1.8]{act} the coproduct of  $\{ X_i : i\in I\}$ is their disjoint union $\dot{\bigcup}_{i\in
I}X_i.$
Likewise, we recall that an $S$-act $A$ decomposable if there
exist two subacts $B,C\subseteq A$ such
that $A=B\cup C$ and $B\cap C=\emptyset$. In this case $A=B\cup C$
is called a decomposition of $A$. Otherwise $A$ is called
indecomposable. By~\cite[theorem 1.5.10]{act}, every $S$-act $A$
has a unique decomposition into indecomposable subacts. We will
use this unique decomposition frequently. For more information
about S-acts we encourage the reader to see \cite{act}.

In this paper, we investigate the cancellation problem in the
category of $S$-acts.
\section{cancellation on \textbf{Act}-$S$}
 In this paper, we give some
 results for cancellation problem in \textbf{Act}-$S$.
 We start with a definition:
\begin{definition}
An $S$-act $A$ satisfies cancellation, if for any $B, C\in
\textbf{Act}-S$ that $A\dot{\cup}B\cong A\dot{\cup}C$ implies
$B\cong C$. If $A$ satisfies cancellation we call $A$ is cancellable.
\end{definition}
There exist examples that cancellation in \textbf{Act}-$S$ always
does not satisfy.
\begin{example}\label{ex}Let $S$ be a monoid.\\
(i) Given two non-isomorphic $S$-acts $B$ and $C$, let
\begin{equation}
    A:=C\dot{\cup}B\dot{\cup} C\dot{\cup}\cdots
\end{equation}
 then
\begin{equation}
 A\dot{\cup}
B\cong A\dot{\cup}C,
\end{equation}
and we can not cancel $A$.\\
(ii)  Take  an indecomposable $S$-act $A$  and an arbitrary infinite
set $I$. Then $B=\dot{\cup}_{i\in I}A_i$ in which $A_i=A$ for any
$i\in I$ is not cancellable, because
\begin{equation}
   B\dot{\cup} A\cong B\dot{\cup}
(A\dot{\cup} A)
\end{equation}
but
\begin{equation}
   A\dot{\cup} A\ncong A.
\end{equation}
\end{example}

\begin{theorem}\label{t:c}
Let $A$ and $B$ are $S$-acts. Then $S$-act $A\dot{\cup} B$ is
cancellable if and only if $A$ and $B$ themselves are.
\begin{proof}
Let $A\dot{\cup} B$ is cancellable and $A\dot{\cup} C\cong
A\dot{\cup} D$ in which $C, D$ are arbitrary $S$-acts. We have
$B\dot{\cup}(A\dot{\cup} C)\cong B\dot{\cup}(A\dot{\cup} D)$
then $(A\dot{\cup} B)\dot{\cup} C\cong (A\dot{\cup}
B)\dot{\cup} D$. Since $A\dot{\cup} B$ is cancellable therefore
$C\cong D$. Hence $A$ is cancellable. In a similar vein, $B$ is
cancellable.\\
Conversely, let $A, B$ are cancellable and $(A\dot{\cup}
B)\dot{\cup} C\cong (A\dot{\cup} B)\dot{\cup} D$ in which $C, D$ are
arbitrary $S$-acts. We have $A\dot{\cup}(B\dot{\cup} C)\cong
A\dot{\cup}(B\dot{\cup} D)$ then $B\dot{\cup} C\cong B\dot{\cup} D$,
because $A$ is  cancellable. As $B\dot{\cup} C\cong B\dot{\cup} D$
and $B$ is cancellable we have $C\cong D$. Therefore $A\dot{\cup} B$
is cancellable.
\end{proof}
\end{theorem}

\begin{proposition}\label{main}
Every indecomposable $S$-act is cancellable.
\end{proposition}
\begin{proof}
Let $A$ be an indecomposable $S$-act and $B$ and $C$ are two arbitrary $S$-acts in which
\begin{equation}
A\dot{\cup}B\cong A\dot{\cup}C.
\end{equation}
We will show that $B\cong C.$ We may assume without loss of
generality that $A\cap B=A\cap C=\emptyset$. Let $f:A\cup
B\longrightarrow A\cup C$ be an $S$-isomorphism. By \cite[Theorem
1.5.10]{act}
we can write $B=\displaystyle\dot{\bigcup_{i\in I}}B_{i}$ and
$C=\displaystyle\dot{\bigcup_{j\in J}}C_{j}$ where
 all  $B_i$'s and $C_j$'s are indecomposable and
$B_i\cap B_{i'}=\emptyset, C_j\cap C_{j'}=\emptyset$ for any $i,
i'\in I$ and $j, j'\in J$. Since $f$ is an isomorphism we get
\begin{equation}
f(A\dot{\cup}(\dot{\bigcup_{i\in
I}}B_{i}))=f(A)\dot{\cup}(\dot{\bigcup_{i\in
I}}f(B_{i}))=A\dot{\cup}(\dot{\bigcup_{j\in J}}C_{j}).
\end{equation}
Here,  by~\cite[Lemma 1.5.36]{act} the subacts $f(A)$ and $f(B_i)$
are indecomposable for any $i\in I$. Furthermore, again
by~\cite[Theorem 1.5.10]{act} this decomposition is unique. Thus,
\begin{equation}
f(A)=A \textrm{ or } f(A)=C_{j'} \textrm{ for  some }  j'\in J.
\end{equation}
 If $f(A)=A$ then for
every $i\in I$ there exists a unique element $j\in J$ such that
 $f(B_i)=C_j$. Therefore
\begin{equation}
B\cong f(B)=\dot{\bigcup_{i\in I}}f(B_{i})=\dot{\bigcup_{j\in
J}}C_j=C,
\end{equation}
because $f$ is an isomorphism. If $f(A)=C_{j'}$ for some $j'\in J$
then $A=f(B_{i'})$ for some ${i'}\in I$. Therefore for every
 $i\neq{i'}$ there exists a unique element $j\neq{j'}$ such that
$f(B_i)=C_j$ and this implies that
\begin{equation}\label{union of f(B)= union of C}
\dot{\bigcup_{i\in I\setminus \{i'\}}}f(B_{i})=\dot{\bigcup_{j\in
J\setminus\{j'\}}}C_j.
\end{equation}
Next from  (\ref{union of f(B)= union of C}) we have
\[\begin{array}{ccl}
  B & \cong & f(B)\\
   &=&\displaystyle f(B_{i'})\dot{\cup}(\dot{\bigcup_{i\in
I\backslash\{{i'}\}}}f(B_{i}))\\
  &=& \displaystyle A\dot{\cup}(\dot{\bigcup_{j\in
J\backslash\{{j'}\}}}C_j) \\
  & \cong  &\displaystyle  f(A)\dot{\cup}(\dot{\bigcup_{j\in J\backslash\{{j'}\}}}C_j)\\
  &=&\displaystyle C_{j'}\dot{\cup}(\dot{\bigcup_{j\in
J\backslash\{{j'}\}}}C_j)\\
&=&C,
\end{array}\]
i.e, $B \cong C$, as required.
\end{proof}
\begin{definition}
Let $A=\dot\bigcup_{i\in I}A_i$ be  the unique
decomposition of $A$ into indecomposable subacts  $A_i, i\in I.$
We call  $A$ finitely decomposable if $1\leq |I|<\infty$. Otherwise
 $A$ is called infinitely decomposable.
\end{definition}
\begin{proposition}\label{fi}
Let $S$ be a monoid. Then\\
(1) Every cyclic $S$-act is cancellable.\\
(2) Every simple $S$-act is cancellable.\\
(3) Every monoid $S$ is cancellable.\\
(4) Every finitely decomposable $S$-act is cancellable.\\
(5) Every finitely generated $S$-act is cancellable.
\end{proposition}
\begin{proof}
Since every cyclic $S$-act is indecomposable by~\cite[Proposition
1.5.8]{act}, (1)-(3) are clear by Proposition~\ref{main}. Statements (4)
and (5) are followed  by Theorem~\ref{t:c} and Proposition~\ref{main}.
\end{proof}
\begin{corollary}\label{fr}
Let $A$ be a free $S$-act with basis $X$. Then $A$ is cancellable if
and only if the basis $X$ is finite.
\begin{proof}
 Since $A$ is free $S$-act by \cite[Theorem 1.5.13]{act}, $A\cong\dot{\cup}_{i\in I}S_i$ where $S_i\cong S$ for any
$i\in I$ and $|I|=|X|$. On the other hand, In Example~\ref{ex} we
have seen that  if $|X|=\infty$ then $A$ is not cancellable.
Therefore by Proposition~\ref{fi} we get the result.
\end{proof}
\end{corollary}
 We have shown that every finitely decomposable $S$-act is cancellable.
The converse is not true in general as the following theorem shows.
\begin{theorem}\label{inf}
Let $A$ be an infinitely decomposable $S$-act such that
$A=\displaystyle\dot{\bigcup_{i\in I}}A_i$ is the unique
decomposition of $A$ into indecomposable subacts and $A_i\ncong A_j$
for any pair of distinct elements $i,j\in I$. Then $A$ is cancellable.
\end{theorem}
\begin{proof}
Assume that $A\dot{\cup}B\cong A\dot{\cup}C$ where $B$ and $C$ are
$S$-act. We must show that $B\cong C.$ We may assume without loss of generality that $A\cap
B=A\cap C=\emptyset$. Let $f:A\cup B\longrightarrow A\cup C$ be an
isomorphism and $B=\displaystyle\dot{\bigcup_{k\in K}}B_k,$
$C=\displaystyle\dot{\bigcup_{j\in J}}C_j$ are unique decompositions
of $B,C$ into their indecomposable subacts, respectively. We have $f(A\cup
B)=A\cup C$ then
\begin{equation}\label{expanded}
\displaystyle\dot{\bigcup_{i\in I}}f(A_{i})\cup(\dot{\bigcup_{k\in
K}}f(B_{k}))=(\dot{\bigcup_{i\in I}}A_{i})\cup(\dot{\bigcup_{j\in
J}}C_{j}).
\end{equation}
Note that by~\cite[Lemma 1.5.36]{act}, in (\ref{expanded})  all
the components on the two sides are indecomposable acts. Since $A_i\ncong A_{i'}$ for any distinct elements $i, i'\in I,$ by applying
\cite[Theorem 1.5.10]{act}, we get for every $i\in I, f(A_i)=A_i$ or $f(A_i)=C_j$ for
some $j\in J.$
 Next Put
\begin{equation}
\begin{split}
I_1&=\{i\in I \mid \ f(A_i)=A_i\}, \\ I_2&=\{i\in I \mid \
f(A_i)=C_j \textrm{ for some } j\in J\},
\end{split}
\end{equation}
\begin{equation}
\begin{split}
J_1&=\{j\in J \mid\ C_j=f(A_i)\textrm{ for  some }i\in I\}, \\
J_2&=\{j\in J \mid \ C_j=f(B_k)\textrm{ for  some }k\in K\},
\end{split}
\end{equation}
 and
\begin{equation}
\begin{split}
K_1&=\{k\in K \mid \ f(B_k)=C_j\textrm{ for  some }j\in J\}, \\
K_2&=\{k\in K \mid \ f(B_k)=A_i\textrm{ for  some }i\in I\}.
\end{split}
\end{equation}
Then it is clear that
\begin{equation}\label{disjoint sets}
I_1\cap I_2=\emptyset, \quad  J_1\cap J_2=\emptyset  \textrm{ and }
K_1\cap K_2=\emptyset
\end{equation}
and
\begin{equation}\label{equipotent sets}
|I_2|=|J_1|,  \quad |K_1|=|J_2| \textrm{ and } |K_2|=|I_2|.
\end{equation}
We have
$\displaystyle\dot{\bigcup_{j\in
J_1}}C_{j}=\displaystyle\dot{\bigcup_{i\in I_2}}f(A_{i})$, because
for any $i\in I_2$ there exists a unique element $j\in J_1$ in such
a way that $f(A_i)=C_j$, and vice versa. Similarly
\begin{equation}
\dot{\bigcup_{j\in J_2}}C_{j}=\dot{\bigcup_{k\in K_1}}f(B_{k})
\textrm{ and } \dot{\bigcup_{i\in I_2}}A_{i}=\dot{\bigcup_{k\in
K_2}}f(B_{k}).
\end{equation}
Now, since $f$ is an isomorphism,  by (\ref{disjoint sets}) and
 (\ref{equipotent sets}) we
obtain
\[\begin{array}{ccl}
C&=&(\displaystyle\dot{\bigcup_{j\in J_1}}C_{j})\cup(\dot{\bigcup_{j\in J_2}}C_{j})\\
 &=& (\displaystyle\dot{\bigcup_{i\in
I_2}}f(A_{i}))\cup(\dot{\bigcup_{k\in K_1}}f(B_{k}))\\
&\cong&(\displaystyle\dot{\bigcup_{i\in
I_2}}A_{i})\cup(\dot{\bigcup_{k\in K_1}}f(B_{k}))\\
&=&(\displaystyle\dot{\bigcup_{k\in K_2}}f(B_{k}))\cup(\dot{\bigcup_{k\in K_1}}f(B_{k}))\\
 &=&f(B).
\end{array}\]
Therefore, $C\cong f(B)\cong B.$
\end{proof}

Let $A$ be an $S$-act and  $A=\displaystyle\dot{\bigcup_{i\in I}}A_i$ be the unique
decomposition of $A$ into indecomposable subacts. Define for $i, j\in
I$, $i\sim j$ if and only if $A_i\cong A_j$. Then $\sim$ is an
equivalence relation on $I$.  The equivalence class $i\in
I$ is given by $[i]=\{j\in I \mid  A_i\cong A_j\}$.

With this introduction we have
\begin{theorem}\label{main1}
Let $A$ be an $S$-act and let $A=\displaystyle\dot{\bigcup_{i\in I}}A_i$ be the unique
decomposition of $A$ into indecomposable subacts $A_i, i\in I$ such that the set
$P=\{{\rm Card} [i] \mid i\in I\}$ is finite.
Then $A$ is
cancellable if and only if the equivalence class $[i]$ is finite for every $i\in I$.
\end{theorem}
\begin{proof} Let $P=\{\text{Card}[i] \mid i\in I\}$ be finite.\\
{\it Necessity.} If for some $i\in I,[i]$ is an infinite set then
$\cup_{j\in [i]}A_j$ is not cancellable (see Example~\ref{ex}).
Therefore by Theorem~\ref{t:c}, $A$ is not cancellable which is a contradiction.\\
{\it Sufficiency.}
 Assume that the equivalence class $[i]$ is finite for every $i\in I$.
 Let $m_1, \cdots ,m_n$ are distinct elements of $P,$ where $m_1$ and $m_n$
 are the smallest and the greatest elements of $P$ respectively.
We define for $1\leq k\leq n$:
\begin{equation}\label{ I_{m_k}}
   I_{m_k}=\{i\in I \mid \text{card} [i]=m_k\}
\end{equation}
  In fact we realize that $I_{m_k}$ is the union of classes that each has $m_k$ elements.  Note that it is possible that $I_{m_k}$ to be infinite. Now we define subsets $I_{m_k}^1, I_{m_k}^2, \cdots ,I_{m_k}^{m_k}$ of $I_{m_k}$  recursively as follows:\\
  Put  $I_{m_k}^1$ to be the subset  $I_{m_k}$ which consists of elements that we choose from  each classes one element. In a similar vein, suppose that we have defined subsets $I_{m_k}^2, \cdots , I_{m_k}^{t-1}$. Then define $I_{m_k}^{t}$ to be the set of elements of $I_{m_k}$ that are not in the earlier  subsets $I_{m_k}^1, I_{m_k}^2, \cdots ,I_{m_k}^{t-1}$. Summarizing  for $1\leq l\leq m_k$ we have,
\begin{equation}\label{}
I_{m_k}^l=\{i\in I_{m_k} \mid i\notin I_{m_k}^{l_1}\textrm{ for
any } l_1<l\}\textrm{ and }[i]\neq [i'],\quad  \forall i\ne i'\in I_{m_k}^l.
\end{equation}
Set
\begin{equation}\label{}
C_{m_k}^{l}=\bigcup_{i\in I_{m_k}^{l}}A_i.
\end{equation}
Then
\begin{equation}\label{}
A=\bigcup_{k=1}^{n}\bigcup_{l=1}^{m_k}C_{m_k}^{l}
\end{equation}
 for $1\leq k\leq n$ and
$1\leq l\leq m_k$. By Theorem~\ref{inf}, $C_{m_k}^{l}$ is cancellable
for every $1\leq k\leq n$ and $1\leq l\leq m_k$, because
for every distinct pair of elements $i,j\in I_{m_k}^{l}$ we have $A_i\ncong A_j$.
Therefore $A$
is cancellable by Theorem~\ref{t:c}.
\end{proof}

\begin{theorem}\label{eq}
Let $A=\dot\bigcup_{i\in I}A_i$ be  the unique
decomposition of $A$ into indecomposable subacts  $A_i, i\in I.$
Furthermore, assume that the set of equivalence classes of $I$, $I/{\sim},$ is finite.
Then  $A$ is cancellable if and only if $A$ is finitely decomposable.

\begin{proof}
Suppose that $A$ is cancellable $S$-act.
If in contrary $A$  is infinitely
decomposable, then $I$ is infinite. Since $|I/{\sim}|<\infty$
there exists an infinite subset $J\subseteq I$ such that $A_i\cong A_j$
for any $i,j\in J.$ Since $\cup_{j\in J}A_j$ is not cancellable (see
Example~\ref{ex}) therefore by Theorem~\ref{t:c}, $A$ is not cancellable, a contradiction.
Therefore  $A$ is finitely decomposable. Then converse is true
by Proposition~\ref{fi}.\\
\end{proof}
\end{theorem}
\begin{remark}
In corollary~\ref{fr} we have seen that a free $S$-act is
cancellable if and only if it is  finitely decomposable. This is
easy by Theorem \ref{eq}, because for each free $S$-act $A$, we have
$A\cong\dot{\cup}_{i\in I}S_i$ which $S_i=S$ for any $i\in I$,
by~\cite[Theorem 1.5.13]{act}, therefore $|I/{\sim}| =1$.
\end{remark}
Let $E(S)$ be the set of all idempotents
of $S.$ By~\cite[Theorem 3.17.8]{act}, an $S$-act $P$ is
projective if and only if $P=\dot{\cup}_{i\in I}P_i$ where $P_i\cong
e_iS$ for idempotents $e_i\in S, i\in I.$  We define an  equivalence relation on
$E(S)$, $e\sim f$ if and only if $eS\cong fS$ which $e,f\in E(S)$.
Let $E(S)/{\sim}=\{[e]; e\in E(S)\}$ where $[e]=\{f\in E(S); fS\cong
eS\}$.
\begin{corollary}\label{pro} Let $S$ be a
monoid in which $|E(S)/{\sim}| <\infty$ and let $P$ be a projective $S$-act. Then
 $P$ is cancellable if
and only if $P$ is finitely decomposable.

\begin{proof}
Since every projective $S$-act is of the form $P\cong\dot{\cup}_{e\in
E(S)}eS,$  the result is clear by Theorem~\ref{eq}.
\end{proof}
\end{corollary}
Here, we introduce the concept of  internal cancellation in \textbf{Act}-$S$. As we have mentiond in the abstract, We shall show
that this coincides with  cancellation.
\begin{definition}
An $S$-act $A$ satisfies internal cancellation if, for any subacts $C, D, E$
and $F$  of $A$, $A=C\dot{\cup} D=E\dot{\cup} F$
and $C\cong E$ implies that $D\cong F$. If $A$ satisfies internal
cancellation we call $A$ is internally cancellable.
\end{definition}
There exist examples that internal cancellation in the category $\textbf{Act}-S$
always does not satisfy. Let us to provide an example.

\begin{example}
Let $S$ be a monoid.
As $S$ with its operation is an $S$-act then all $S$-acts
$$C=S\times
\{1\}, D=\bigcup_{i\in\mathbb{N}}(S\times{\{i+1\}}), E=S\times
\{1\}\cup S\times \{2\},
F=\bigcup_{i\in\mathbb{N}}(S\times{\{i+2\}})$$
 with actions induced by the action of $S$ are $S$-acts.
 Furthermore, we have $C\cup D=E\cup F$ and $D\cong F$, but $C\ncong E$.
 It is means that, the $S$-act $A=\displaystyle\bigcup_{i\in\mathbb{N}}(S\times\{i\})$
 is not internally cancellable.
\end{example}

\begin{theorem}\label{ic}
Let $A$ be an $S$-act. Then $A$ is cancellable if and only if $A$ is
internally cancellable.
\end{theorem}
\begin{proof}
{\it Necessity.} Assume  $A=C\dot{\cup} D=E\dot{\cup} F$ in which
$C, D, E, F$ are subacts of $A$ and $C\cong E$. Then  $C\dot{\cup}
D\cong C\dot{\cup} F$. By Theorem~\ref{t:c}, $C$ is cancellable and
then $D\cong F.$ Therefore $A$ is internally
cancellable.\\
{\it Sufficiency.} Suppose $A$ is an internally cancellable $S$-act
and $A\dot{\cup} B\cong A\dot{\cup} C$ in which
$B,C\in\textbf{Act}-S.$ We may assume without loss of generality
that  $A\cap B=A\cap C=\emptyset$. Let $f:A\cup B\longrightarrow
A\cup C$ be an $S$-isomorphism. Since $f(A)\cup f(B)=A\cup C$,
intersect this equation once with $f(A)$ and once more with $A$ we
get
\begin{equation}\label{f(A)}
f(A)=(A\cap f(A))\cup(f(A)\cap C)
\end{equation}
 and
 \begin{equation}\label{A}
 A=(f(A)\cap A)\cup(f(B)\cap A).
 \end{equation}
 Combine together equations (\ref{f(A)}) and (\ref{A}) gives us
 \begin{equation}
  A=f^{-1}(A\cap
f(A))\cup f^{-1}(f(A)\cap C)=(f(A)\cap A)\cup(f(B)\cap A),
\end{equation}
 Since
 \begin{equation}
 f^{-1}(f(A)\cap A)\cong f(A)\cap A
 \end{equation}
   and $A$ is internally
cancellable we deduce that
 \begin{equation}
 f^{-1}(f(A)\cap C)\cong f(B)\cap A.
\end{equation} i.e.,
\begin{equation}\label{f^{-1}(C)}
 A\cap f^{-1}(C)\cong f(B)\cap A.
\end{equation}
Since $f^{-1}$ is an isomorphism we get
\begin{equation}
 f^{-1}(f(B)\cap C)\cong f(B)\cap C
\end{equation}
i.e.,
\begin{equation}\label{f(B)}
 B\cap f^{-1}(C)\cong f(B)\cap C.
\end{equation}In a similar way, as we did in~(\ref{f(A)}) we have
\begin{equation}\label{f^{-1}(C) expanded}
 f^{-1}(C)=(f^{-1}(C)\cap A)\cup(f^{-1}(C)\cap B)
\end{equation}
 and
\begin{equation}\label{f(B) expabded}
 f(B)=(f(B)\cap A)\cup (f(B)\cap C).
\end{equation}
As $f$ and $f^{-1}$ are isomorphism we have $B\cong f(B), C\cong
f^{-1}(C)$. Now by (\ref{f^{-1}(C)}), (\ref{f(B)}),  (\ref{f^{-1}(C)
expanded}) and (\ref{f(B) expabded})  we deduce $f^{-1}(C)\cong
f(B)$ and so $B\cong C.$
Therefore $A$ is cancellable.
\end{proof}

\section*{Acknowledgment(s)}
The authors would like to thank Semnan university for its
financial support.


\end{document}